\documentclass[11pt]{article}

\usepackage[margin=0.85in]{geometry}
\usepackage{amsmath,amsfonts,amscd,amsthm}
\usepackage{subfigure}
\usepackage{enumerate}
\usepackage{cite}
\usepackage{mathrsfs}
\usepackage{graphicx}
\usepackage{srcltx}
\usepackage{color}
\usepackage{rotating}
\usepackage[usenames,dvipsnames]{pstricks}
\usepackage{mathtools}

\newtheorem{theorem}{Theorem}[section]

\newtheorem{proposition}[theorem]{Proposition}
\newtheorem{corollary}[theorem]{Corollary}
\newtheorem{definition}[theorem]{Definition}

\newtheorem{remark}[theorem]{Remark}
\newtheorem{algthm}[theorem]{Algorithm}

\newcommand{\C}{\mathbb{C}} 
\newcommand{\G}{\Gamma}
\newcommand{\z}{\mathbf{Z}}
\newcommand{\SO}{\mathbf{SO}}
\newcommand{\On}{\mathbf{O}}
\newcommand{\D}{\mathbf{D}}
\newcommand {\eq}{\begin{equation}}
\newcommand {\feq}{\end{equation}}
\newcommand{\Geta}{\Gamma_{\eta}}
\newcommand{\Gmais}{\Gamma_{+}}

\newcommand{\p}{\mathcal{P}}

\newcommand{\R}{\mathbb{R}}
\newcommand{\pg}{\mathcal{P}(\Gamma)}
\newcommand{\bg}{\mathcal{P}}

\newcommand{\pgmais}{\mathcal{P}(\Gamma_{+})}
\newcommand{\pgflexa}{\overrightarrow{\mathcal{P}}}
\newcommand{\bgflexa}{\overrightarrow{\mathcal{P}}}
\newcommand{\pgflexamais}{\overrightarrow{\mathcal{P}}(\Gamma_{+})}

\newcommand{\pvwflexa}{\overrightarrow{\mathcal{P}}(\Gamma)}
\newcommand{\qg}{\mathcal{P}^{\eta}(\Gamma)}
\newcommand{\qvwflexa}{\overrightarrow{\mathcal{P}}^{\eta}(\Gamma)}
\newcommand{\pvwflexamais}{\overrightarrow{\mathcal{P}}(\Gamma_{+})}

\newcommand{\Rflexa}{\overrightarrow{R}}
\newcommand{\Sflexa}{\overrightarrow{S}}
\newcommand{\Rdflexa}{\overrightarrow{R}}
\newcommand{\Sdflexa}{\overrightarrow{S}}

\title{On equivariant binary differential equations}

\author{Miriam Manoel \ and \ Patr\'icia Tempesta \\	 Mathematics Department ICMC, \\ University of  S\~ao Paulo, S\~ao Paulo, Brazil \\ e-mail\textup{: \texttt{miriam@icmc.usp.br}} and \textup{\texttt{ptempest@icmc.usp.br}}}

\date{}

\begin{document}

\maketitle

\begin{abstract}
	This paper introduces the study of  occurrence of  symmetries in binary differential equations (BDEs). These  are implicit differential equations given by the zeros of a quadratic 1-form, $a(x,y)dy^2 + b(x,y)dxdy + c(x,y)dx^2 = 0,$ for $a, b, c$  smooth real functions defined on an open set of $\mathbb{R}^2$. Generically, solutions of a BDE are given as leaves of a pair of foliations, and the action of a symmetry must depend not only whether it preserves or inverts the plane orientation, but also whether it preserves or interchanges the foliations. The first main result reveals this dependence, which is given algebraically by  a formula relating three group homomorphisms defined on the symmetry group of the BDE.  The second main result adapts algebraic methods from invariant theory for representations of compact Lie groups on the space of quadratic forms on  $\R^n, n \geq 2$. With that we obtain an algorithm to compute general expressions of quadratic forms. Now, symmetric quadratic 1-forms are in one-to-one correspondence with equivariant quadratic forms on the plane, so these are treated  here as a particular case. The general forms of equivariant quadratic 1-forms under each compact subgroup of the orthogonal group ${\mathbf{O}(2)}$ are also given.
\end{abstract}

\begin{flushleft}
	
\end{flushleft}

\section{Introduction} \label{sec:Intro}
A binary differential equation on the plane, or a BDE, is an implicit quadratic differential equation
\begin{equation}\label{BDE}
a(x,y) dy^2 + 2b(x,y)dx dy + c(x,y)dx^2 =0,
\end{equation}
where the coefficients  $a, b, c$ are real functions which we assume to be smooth on an open set  $U \subseteq {\R}^2$.  The function $\delta: U \rightarrow \R, $ $\delta(x,y)= (b^2 - ac) (x,y)$, is the  {\it discriminant function} and its zero set
\[ \Delta=\{(x,y) \in U: (b^2 - ac) (x,y)=0 \} \]
is the {\it discriminant set } of the BDE. 
The investigation of occurrence of symmetries is converted in purely algebraic terms, so there is no loss of generality in assuming that $U$ is the whole plane, which we do from now on for simplicity.
At points  where $\delta>0$,  (\ref{BDE}) defines a pair of transversal directions, and by the {\it configuration} associated with this equation we mean the distribution of all  solution curves tangent to these directions. The geometry of a BDE configuration is a subject of great interest, with important applications in differential geometry,  as the equations of lines of curvatures, characteristic curves and asymptotic curves of smooth surfaces, and in control theory (see \cite{tari} and references therein). Conditions for local stability of positive  binary differential equations ($\delta >0$) and their classification are given in \cite{guinez2} and \cite{guinez1}, with a  description of the  topological patterns that bifurcate in one-parameter families of these equations.  Singular points of a  class of positive binary equations  associated with a smooth surface are also studied in\cite{guinez2} and \cite{sotoguitierrez}, the coefficients of BDE  being given in terms  of the coefficients of the first and second fundamental forms of the surface.
Determining  models of configurations associated with BDEs has been  also addressed in many works; for example, in \cite{brucetari, brucetari2, brucetari3, dara, gutierrezguinez, tari2, tari3}  the classification of  BDEs is performed up to topological, formal, analytic or smooth equivalences. We refer to \cite{tari} for a survey on these topics.

\begin{figure}[!h]
	\centering
	\includegraphics[scale=0.5]{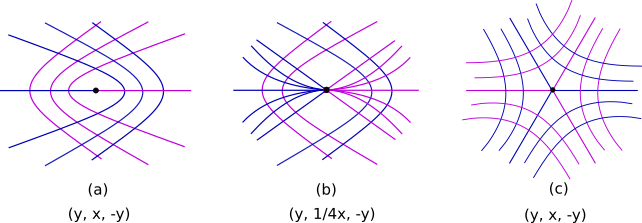}
	\caption{Configurations of symmetric BDEs. In (a) and (b) the symmetry group is $\z_2\times \z_2$ and  in (c) the symmetry group is  $\D_6$. }
	\label{Figura 1}
\end{figure}

This work is motivated by the recognition of symmetries in most normal forms that appear in the works mentioned above. In  the  linear case for example, namely when the coefficients of (\ref{BDE}) are linear functions, presence of at least one nontrivial symmetry (namely minus identity) is a necessary condition. As a consequence, our study shows which symmetry groups are attained or never attained in the nonlinear cases.  The purpose of the present paper is to introduce the systematic study of symmetries in binary differential equations.  \cite{ManTemp1} is a continuation of this study: we  investigate symmetries of a BDE in connection with an associated pair of equivariant vector fields  (see Remark~\ref{rmk: 1-forms}); also, quadratic forms with homogeneous coefficients are studied through an analysis of the number of invariant lines that appear in the configuration space imposed by their group of symmetries; in addition, we analyse possible symmetry groups of BDEs with  Morse type discriminant introduced in \cite{brucetari}. 
More recently we have driven attention to some questions relating symmetries in differential forms to pairs of foliations of special classes of surfaces; for example, for a given equivariant BDE, we ask whether this can be realized as an equation of lines of curvatures or of asymptotic lines of a surface immersed in ${\R}^n$ for some $n$.

To remark on evidences of occurrence of symmetries in configurations associated with BDEs, let us consider the configurations that appear in  \cite{brucefidal}, which are  generic topological structures of principal direction fields at umbilic points of surfaces on Euclidean spaces.  The normal forms are given as triples $(a, b,c)$ for  $c=\pm a$ in (\ref{BDE}) and their  configurations  are reproduced  in  Fig.~\ref{Figura 1}, the so-called (a) lemon, (b) monstar and (c) star.  The solution curves determine two foliations of the plane distinguished by colour, and the black point is the discriminant set.
The pictures clearly suggest an invariance of  the three configurations under reflection with respect to the $x$-axis. There is another invariance with respect to the $y$-axis, which is given by this operation followed by a change of colour. As a consequence, the composition of these two elements (minus identity) must be a symmetry which interchanges colour. In fact, we should recognize {\it a priori}  minus identity  in the set of symmetries of all these cases by their linearity, as mentioned above.  The third picture has also six rotational symmetries, three of which are colour-preserving and  the other three are colour-interchanging.  In fact, the full symmetry group of pictures (a) and (b) is ${\bf Z}_2 \times {\bf Z}_2$, generated by the reflections across the axes,  and the full symmetry group of picture (c) is the dihedral group  $\D_6$, generated by a reflection and a rotation of order six. As these examples illustrate, the group action must be defined taking colour changes into consideration at the region on the plane where (\ref{BDE}) defines a bivalued direction field. As we shall see,  the  action  of a symmetry group $\G$ of a binary differential equation must be defined on the tangent bundle through its representation on the plane combined with a group homomorphism $\eta : \G \to {\bf Z}_2 = \{\pm 1\}$ (Definition~\ref{def:symmetric BDE}). This action is translated into an action of $\Gamma$ on the space of symmetric matrices under conjugacy. This  is then used to obtain the general forms of equivariant BDEs under all compact subgroups of ${\bf O}(2)$ through invariant theory. 

The paper is organized as follows: in Section~\ref{sec: symmetry action} we introduce the notion of  symmetries in a  BDE, namely  when the equation is invariant under the linear action of a subgroup $\Gamma \leq {\bf O}(2)$. We formalize the concept using group representation theory on the tangent bundle on which the associated quadratic 1-form is defined. One of our two main results is Theorem~\ref{thm: homomorphisms formula}, which establishes a formula that reveals the effect of a symmetry in the configuration geometry  in simple algebraic terms.  In Section \ref{sec: inv theory} we generalize the results in \cite{patricia} for  $\G$-equivariant mappings with distinct  actions on the  source and target. The  results allow the computation of general forms of  equivariant mappings using an algebraic algorithm (Algorithm~\ref{algoritmo das apl geta equivariantes}). In Section~\ref{sec: general forms} we implement  these results  to calculate the general forms of equivariant BDEs under any compact subgroup of ${\bf O}(2)$, which is our second main result, summarized  in Table~1; in this section we also present a number of examples.

\section{Symmetry groups of binary differential equations} \label{sec: symmetry action}

In this section we formalize the concept of a symmetric  binary differential equation  under the linear action of a  compact Lie subgroup $\G$ of  $\On(2)$ .

Let ${\cal F}({\R}^2)$ denote the set of real $C^\infty$ quadratic differential forms $\omega$ on  $M= {\R}^2$, $\omega : TM \to {\R} $  defined on the tangent bundle
$TM = \sqcup_{u \in M} T_u M$ given by
\eq \label{eq:1-form} \omega  = a(x,y) dy^2 + 2b(x,y) dx dy + c(x,y) dx^2,
\feq
with $a,b,c$  $C^\infty$ functions on ${\R}^2$.
Let $\G$ be a compact  Lie group acting linearly on $\R^2$.  This induces an action of $\G$ on the tangent bundle,
\[ \begin{array}{ccc}
\G \times \sqcup_{u \in M} T_u M & \to & \sqcup_{u \in M} T_u M \\
(\gamma, u, v) &\mapsto & \gamma \cdot (u,v) \ = \ \bigl(\gamma u, (d \gamma\bigr)_u (v) ).
\end{array} \]
Since the action is linear, $\gamma \cdot (u,v) =(\gamma u, \gamma v )$.

Let $\eta:~\G~\rightarrow~\z_2~=~\{\pm 1\}$  be a one-dimensional representation of $\G$, acting on the target of $\omega$. We then give the following:

\begin{definition} \label{def:equivariant}
	An element $\omega \in {\cal F}({\R}^2)$ is $\G_\eta$-equivariant if, for all $\gamma \in \G$,
	\begin{equation}
	\label{eq:equivariant form}\omega(\gamma \cdot (x,y,dx,dy)) \ = \ \eta(\gamma) \omega(x,y,dx,dy).
	\end{equation}
\end{definition}

The BDE (\ref{BDE})  is the equation of zeros of (\ref{eq:1-form}), so we define

\begin{definition} \label{def:symmetric BDE}
	For $\G$ a compact Lie group acting linearly on ${\R}^2$ and $\eta:~\G~\rightarrow~\z_2$ a group homomorphism, the binary differential equation $(\ref{BDE})$ is $\G$-invariant, or $\G$ is the symmetry group of $(\ref{BDE})$, if $\omega$ in $(\ref{eq:1-form})$ is $\G_\eta$-equivariant.
\end{definition}

We notice  that the group of symmetries of a BDE generally admits, by its nature, an order-2 normal subgroup, which is equivalent to saying that the group homomorphism $\eta$ in Definition \ref{def:symmetric BDE} is nontrivial.   Hence, in the present context, symmetries are established through the group together with this homomorphism, and hence this action is denoted below by $\G_{\eta}$.

\begin{remark} \label{rmk: inclusion1}
	Let $\Sigma(\Delta) \leq {\bf O}(2)$ denote  the group of symmetries of $\Delta$, namely, the subgroup of elements of  ${\bf O}(2)$ that leave $\Delta$ setwise invariant. Then
	\eq \Gamma_{\eta} \leq \Sigma(\Delta).\label{eq: inclusion}\feq
	In other words, symmetries of a BDE are at most the symmetries of the discriminant set. This can be of practical use when detecting the symmetry group of the equation if we know the shape of $\Delta$.
\end{remark}

Solutions of (\ref{BDE})  are nonoriented curves,  associated with direction fields. At the region on the plane where the discriminant $\delta$ is positive, these form a pair of transverse foliations ${\mathscr F}_1$ and ${\mathscr F}_2$. The set of symmetries of each foliation, that is, the set of elements that leave them setwise invariant  is a subgroup of $\G$. The structure of this subgroup is discussed below.  This pair is, in turn, associated with  two oriented foliations given as integral curves of the vector fields
\eq \label{eq:vector fields}
X_i(x,y) \ = \ (a(x,y), b(x,y) +(-1)^i \sqrt{\delta(x,y)}), \ \ i=1,2.
\feq
Consider for example the equation
\[ \omega(x,y) = y dy^2 + 2x dx dy - y dx^2=0, \]
for which the associated vector fields are
\[ X_i(x,y) \ = \ (y, x +(-1)^i \sqrt{x^2 + y^2}), \ \ i=1,2. \]
The configuration of this equation is given in Fig.~\ref{Figura 1}(a). Both $X_1$ and $X_2$ are reversible-equivariant vector fields under the action of the group $\z_2$ generated by the reflection $\kappa_x$ with respect to the $x$-axis,
\[ X_i(\kappa_x (x,y)) \ = \ - \kappa_x  X_i(x,y), \\ i =1,2. \]
As the picture suggests, this reflection is in fact a symmetry of the BDE. Now, the combination of the two foliations  adds symmetries to the whole picture, leading to a configuration which is also symmetric with respect to the reflection on the $y$-axis. By the nature of this additional symmetry, this element should invert foliations. In fact, we prove that $\z_2 \times \z_2$ is the symmetry group of the BDE.   However, it is more subtle to realize how each element should act on the form $\omega$, in the sense that it is not obvious whether $\eta(\gamma) =1$ or $-1$ for each $\gamma$ in the group. As we shall see in Theorem~\ref{thm: homomorphisms formula}, this depends not only whether $\gamma$  preserves or interchanges foliations, but also whether it preserves or inverts orientation on the plane.  \\

Before we state the result, we introduce another homomorphism: Consider the open region  in ${\R} ^2$
\[ \Omega  =  \{ (x,y) \in {\R} ^2 \ : \ \delta(x,y) > 0\}, \]
and consider the restriction of the action of $\G_\eta$ on $\Omega$. This is well-defined since the discriminant set $\Delta$ is $\G$-invariant and splits the plane into two invariant regions, where $\delta$ is positive or negative. For BDEs (\ref{BDE}) for which $\Omega$ is non-empty, we introduce the homomorphism $\lambda : \G \to \z_2 = \{\pm1\}$,
\begin{equation} \label{eq:lambda}
\lambda (\gamma) = \left\{  \begin{array}{cl}
1, & \gamma({\mathscr F}_i) = {\mathscr F}_i \\
-1, & \gamma({\mathscr F}_i) = {\mathscr F}_j, \ j \neq i, \end{array} \right.
\end{equation}
$i,j \in \{1,2\}$. It follows directly from this definition that the subgroup of symmetries of each foliation ${\mathscr F}_i$, $i=1,2$,  is
$\Sigma({\mathscr F}_i) = \ker \lambda$.

\begin{theorem} \label{thm: homomorphisms formula}
	Let \ $\eta, \ \lambda : \G \to \z_2$ be the two group  homomorphisms of Definition~$\ref{def:symmetric BDE}$ and of $(\ref{eq:lambda})$. Then, for all $\gamma \in \G$,
	\[ \lambda (\gamma) = \det(\gamma) \ \eta(\gamma). \]
\end{theorem}

\begin{proof}
	At  $(x,y) \in \Omega$ consider the pair of tangent vectors given in (\ref{eq:vector fields}),
	$$  X_i(x,y)=(a(x,y), -b(x,y) + (-1)^i\sqrt{\delta(x,y)}), \ i=1,2.$$
	From the definition of the action $\G$ on $TM$,  the pair of tangent vectors to the two solution curves at $\gamma (x,y)$ is given by $\gamma X_i(x,y)$, $i=1,2$. On the other hand, from the equivariance of $\omega$ under $\Geta$, the vectors
	$$X_i(\gamma(x,y))=(\eta a(x,y), -\eta b(x,y) + (-1)^i \sqrt{\delta(x,y)}), \ i=1,2,$$
	are also tangent vectors to the two solution curves at $\gamma (x,y)$. By symmetry it follows that these two pairs must be parallel,  i.e., there exists a nonzero $\alpha$ such that, for $i \neq j \in \{1,2\},$
	\eq \label{igualdades}
	\gamma X_i (x,y)= \left \{ \begin{array}{ll}
		\alpha X_i(\gamma(x,y)),  &  \hbox{if} \   \lambda(\gamma))=1  \\
		\alpha X_j(\gamma(x,y)),  &  \hbox{if} \   \lambda(\gamma))=-1.
	\end{array} \right.
	\feq
	Also, by the orthogonality of the action, we have $\alpha = \pm1$.
	Now, consider the two matrices $M_1$ and $M_2$ whose columns are  the vectors $\gamma X_1, \gamma X_2$ and $X_1\gamma, X_2\gamma$  both calculated at $(x,y)$, respectively, that is,
	$$M_1=\left(
	\begin{array}{cc}
	\gamma X_1 & \gamma X_2 \\
	\end{array}
	\right),
	\ M_2=\left(
	\begin{array}{cc}
	\eta(\gamma)a &  \eta(\gamma)a\\
	-\eta(\gamma)b- \sqrt{\delta}& -\eta(\gamma)b+ \sqrt{\delta} \\
	\end{array}
	\right)
	$$
	From  (\ref{igualdades}) it follows that
	\[  \det(M_1)= \alpha^2 \lambda(\gamma) \det(M_2)  =  \lambda(\gamma) \det(M_2). \]
	Finally,  $\det(M_1) = \det(\gamma)  \ 2a\sqrt{\delta} $, and $\det(M_2)= \eta(\gamma) 2 a \sqrt{\delta}$. Hence,
	\[  \det(\gamma) = \eta(\gamma) \lambda(\gamma), \]
	which implies the result since these are all group homomorphisms $\G \to \z_2=\{\pm 1\} $.
\end{proof}

\begin{remark}
	In practice, Theorem $\ref{thm: homomorphisms formula}$ adds information to the inclusion $(\ref{eq: inclusion})$ when detecting the symmetry group  $\G_{\eta}$ of a BDE. In fact,  it provides the construction of the  homomorphism $\eta$  by the geometrical investigation of whether each element $\gamma \in \G_\eta$ preserves the foliations ($\lambda(\gamma) = 1$) or interchanges the foliations  ($\lambda(\gamma) = -1$).  To illustrate, consider the pictures in Fig.~1. In $(a)$ and $(b)$, foliations are interchanged by $\kappa_y$, whereas they are preserved by $\kappa_x$, and now we use $\det(\kappa_x) = \det(\kappa_y) = -1$ to conclude by Theorem $\ref{thm: homomorphisms formula}$ that  $\eta(\kappa_x) =  - \eta(\kappa_y) = -1$.  These are the generators of the symmetry group $\z_2 \times \z_2$, and  so the homomorphism $\eta$ is well-established for these examples. In $(c)$ foliations are interchanged by $\kappa_y$ and  rotation of $\pi/3$; since these are orientation reserving and orientation preserving respectively, it follows that   $\eta$ assumes $1$ and $-1$, respectively. These are the generators of the symmetry group 
	${\bf D}_6$, and so the homomorphism $\eta$ is well-established. 
\end{remark}

The following property is direct from Theorem \ref{thm: homomorphisms formula}:
\begin{corollary} \label{cor: homo formula}
	$\ker \lambda \cap \ker \eta $ is a cyclic group of $\Gamma$.
\end{corollary}

A  quadratic differential form (\ref{eq:1-form})  is associated to the matrix-valued mapping $B: \mathbb{R}^2 \rightarrow Sym_2,$
\eq  \label{forma quadratica} B(x,y)=\left(
\begin{array}{cc}
	c(x,y) & b(x,y) \\
	b(x,y) & a(x,y) \\
\end{array}
\right), \feq
where $Sym_2$ denotes the set of  symmetric matrices of order 2. Then  (\ref{eq:1-form}) can be written as
\[ \omega \ = \ \left(
\begin{array}{c}
dx\\
dy\\
\end{array}
\right)^t  \ B(x,y) \  \left(
\begin{array}{c}
dx \\
dy \\
\end{array}
\right), \]
where superscript $t$ denotes transposition. From (\ref{eq:equivariant form}), it follows that   symmetries of (\ref{BDE}) are given by the following equivariance condition of $B$,  with the action on the target  given by the homomorphism $\eta$ and conjugacy:
\begin{equation}  
B\bigl(\gamma (x,y)\bigr) = \eta(\gamma)  \gamma B(x,y) \gamma^t,  \ \forall \ \gamma \in \Gamma. \label{relacao}
\end{equation}

\begin{remark} \label{rmk: 1-forms}
	Although  Definition $\ref{def:equivariant}$  is given here for quadratic 1-forms, the equivariance condition $(\ref{eq:equivariant form})$  can be given for general 1-forms. For the linear case, it follows that 
	\[ \alpha(x,y, dx, dy) = A(x,y) dy + B(x,y) dx \]
	is $\G_\eta$-equivariant if, and only if, the associated planar mapping $F = (A, B)$ is 
	$\G_{\eta}$-equivariant in the sense of Definition $\ref{def:inv and equiv}$  below. For an investigation of pairs of symmetric linear 1-forms and their connection with symmetries of an associated BDE we refer to \mbox{\cite{ManTemp1}}. 
\end{remark}

From now on we shall use the matricial notation (\ref{relacao}) to investigate symmetries in BDEs. For a given $\G_\eta$, the aim is to find the general form of symmetric matrices (\ref{forma quadratica}) satisfying (\ref{relacao}). This is an algebraic problem which is treated in  Section \ref{sec: general forms}. \\

\section{\label{sec: inv theory}Invariant theory for the group $\Geta$}

The aim of this section is to obtain generalizations of some results of \cite{patricia} to be applied to a systematic study of symmetries in BDEs. The idea is to use group representation theory in the study of symmetries in the space of quadratic forms.  In Subsection~\ref{subsec:generalization} we register the generalization of results of \cite{patricia} to $\Geta$-equivariant mappings $V \to W$ (possibly distinct source and target), whose proofs adapt straightforwardly from \cite{patricia}.
In Subsection~\ref{subseq: equiv quadratic forms} we restrict the results of the preceding subsection to equivariant mappings ${\R}^n \to Sym_n$, where $Sym_n$ denotes the space of symmetric matrices of order $n$. These are then used in  Section \ref{sec: general forms} in the study of symmetries in binary differential equations.

\subsection{$\G_{\eta}$-equivariant mappings} \label{subsec:generalization}

We start with a brief summary about  algebraic invariant theory of compact Lie groups. Throughout we consider a compact Lie group $\G$ acting linearly on a real vector space $V$ of finite dimension $n$. There exists a $\G$-invariant  inner product on $V$ under which the representation $(\rho, V)$ associated with the given action of $\G$ is orthogonal, {\it i.e.}  for $\gamma \in \G$, $\rho(\gamma) \in \On(n)$, the group of orthogonal matrices of order $n$ (\cite[XII, Proposition 1.3]{golubitsky}). Hence, Lie groups in this paper are the closed subgroups of $\On(n)$.
A real function $f:V \rightarrow \R$ is {\it $\G$-invariant} if
$$\label{fc invariant} f(\rho(\gamma)x)= f(x), \ \ \forall \ \gamma \in \G, \ \forall \ x \in V. $$

The set $\pg$ of $\G$-invariant polynomials is a ring over $\R$. A finite set $\{u_1, ... u_s\}$ of $\G$-invariants generating this ring is called a { \it Hilbert basis} for $\pg.$ The existence of a Hilbert basis was proved by Weyl in $1946$, and Schwarz proved in $1975$ that the same set generates the ring of $C^{\infty}$ $\G$-invariant germs (see \cite{golubitsky}).

For $(\rho, V)$ and $(\nu, W)$ representations of $\G,$  a mapping $g: V \rightarrow W$ is  {\it $\G$-equivariant} if
$$ \label{apl equivariante} g(\rho(\gamma))=\nu(\gamma)g(x), \ \ \forall \ \gamma \in \G, \  \forall x \in V. $$

The set $\pvwflexa$  of $\G$-equivariant mappings $V \to W$ with polynomial entries is a module over $\pg$. Po\'enaru in 1976 proved that $\pvwflexa$ is finitely generated over the ring $\pg$ and  that the same set generates the module of equivariant  $C^{\infty}$~germs over the ring of invariant germs.

We shall also consider a one-dimensional representation of $\G$,
\eq \label{homom eta} \eta: \Gamma \rightarrow \z_2=\{\pm 1\}. \feq

This is a group homomorphism with $\Gmais = \ker \eta$  a normal subgroup of $\G$ of index $2$ if $\eta$ is nontrivial. From that, we also have the so-called  $\eta$-{\it dual representation} of $(\nu, W)$, denoted by $\nu_\eta$, defined by the product
$$ \label{dual} \gamma \mapsto \nu_\eta(\gamma) = \eta(\gamma) \nu(\gamma).$$


\begin{definition} \label{def:inv and equiv}
	For  $\eta$ as in $(\ref{homom eta})$ and $(\rho, V)$ and $(\nu, W)$ representations of $\G$,  a function $f: V \rightarrow \R$ is  $\G_{\eta}$-invariant if
	$$\label{fc eta invariant} f(\rho(\gamma)x)=\eta(\gamma)f(x), \ \ \forall \ \gamma \in \G, \ \forall \ x \in V. $$
	A mapping $g:V \rightarrow W$ is  $\G_{\eta}$-equivariant if
	$$\label{ap eta eqvariant} g(\rho(\gamma)x)=\eta(\gamma)\nu(\gamma)g(x), \ \ \forall \ \gamma \in \G, \ \forall \ x \in V. $$
\end{definition}
We shall denote
$\qvwflexa$ and $\qg$ the sets of $\G_{\eta}$-equivariant polynomial mappings and $\G_{\eta}$-invariant polynomial functions, respectively.  A $\G_{\eta}$-equivariant is an equivariant $(\rho, V) \to (\nu_{\eta}, W)$, and  a $\G_{\eta}$-invariant  is an equivariant $(\rho, V) \to (\eta, \R)$, so the finitude of generators for each as $\pg$-modules follows by Po\'enaru's theorem. It also follows from this theorem that it is no loss of generality to restrict to modules of polynomials when finding generators of $C^\infty$ mappings. In Section~\ref{sec: general forms} we shall see how configurations change when the symmetry group of binary differential equations changes to its dual representation. \\

A connection is established in \cite{patricia}  between the invariant theory for $\G$ and  $\Gmais$. This is done  through  an algebraic algorithm to compute generators of  $\qvwflexa$ from the knowledge of generators of  $\pvwflexamais$, when source and target spaces are the same. In Proposition~\ref{prop: generalization}  and Algorithm~\ref{alg} we  generalize this,  with  a similar algorithm to compute generators of $\Geta$-equivariants with possibly distinct source and target.

We follow the notation used in \cite{patricia} to introduce the Reynolds operators  $R:~\pgmais ~\rightarrow~\pgmais$ and $\Rflexa : \pvwflexamais \rightarrow \pvwflexamais$,
$$\label{reynods r e s} R(f)(x)=1/2(f(x) + f(\rho(\delta) x))  \ \ \mbox{e}  \ \  \overrightarrow{R}(g)(x) =1/2(g(x) + \nu(\delta)^{-1}g(\rho(\delta) x)), $$
and, the $\eta$-Reynolds operators on $\pgmais$ and on $\pvwflexamais$, $S: \pgmais \rightarrow \pgmais $ and $\Sflexa : \pvwflexamais \rightarrow \pvwflexamais$,
$$ {S}(f)(x)=1/2(f(x) - f( \rho( \delta ) x)) \ \ \mbox{e} \ \ \overrightarrow{S} (g)(x)= 1/2 (g(x) - \nu(\delta)^{-1}g(\rho(\delta) x)), $$
for an arbitrary fixed $\delta \in \G \setminus \G_{+}$.

Let us denote by $I_{\pgmais}$ and $I_{ \small \pvwflexamais}$ the identity maps on $\pgmais$ and on $\pvwflexamais$, respectively.
\begin{proposition} \label{prop: generalization} The operators above satisfy the following:
	\begin{description}
		\item[(a)] They are homomorphisms of $\pg$-modules and
		$$ R + S = I_{\pgmais} \ \ \ \mbox{and} \ \ \  \Rflexa + \Sflexa= I_{ \small \pvwflexamais}.$$
		\item[(b)] They are idempotent projections and the following direct sum decompositions of $\pg$-modules hold:
		\eq \label{decomposition} \pgmais = \pg \oplus \qg \ \ \ \mbox{and} \ \ \ \pvwflexamais = \pvwflexa \oplus \qvwflexa. \feq
	\end{description}
\end{proposition}
\begin{proof}
	 Analogous to Propositions 2.3 and  2.4 in \cite{patricia}. 
\end{proof}

The algorithm is based on the decompositions  (\ref{decomposition}) and on the projection operators  $S$ and $\Sflexa$ applied to  a given Hilbert basis of $\pgmais$ and a set of generators of $\pvwflexamais$. The  procedure is:
\begin{algthm} \label{alg}
	\label{algoritmo das apl geta equivariantes}
	Let $\G$ be a closed subgroup of $\On(n)$ and  $\eta: \Gamma \rightarrow \mathbf{Z}_2$ a homomorphism with $\ker \eta =\Gamma_{+}$,  $\{ u_1,...,u_s\} $  a Hilbert  basis of $\pgmais$ and $\{H_0,...,H_r\} $ a  generator set of $\pvwflexamais$ as a $\pgmais$-module;
	\begin{description}
		\item[1]  Fix $\delta \in \Gamma \setminus \Gamma_{+}$ arbitrary;
		\item[2]  For $i \in \{1, ..., s\}$, do $\tilde{u}_i=S(u_i), \tilde{u_0}:=1;$
		\item[3]  For $i \in \{1, ..., s\}$ and $j \in \{0, ..., r\}$, do $H_{ij}=\tilde{u}_i H_j;$
		\item[4]  For $i \in \{1, ..., s\}$ and $j \in \{0, ..., r\}$, do $\tilde{H}_{ij}=\Sflexa(H_{ij})$.
	\end{description}
	{\bf Result:}  $\{\tilde{H}_{ij}: \ 0 \leq i \leq s, \: 0 \leq j \leq r\}$ is a generator set of $\qvwflexa$ as a $\pg$-module.
\end{algthm}

As proved in \cite{patricia},  step 2 above provides a generator set of the $\pg$-module  $\qg$ (these are the anti-invariants in that paper). What we also remark at this point is that replacing $\Sflexa$ by the projection operator  $\Rflexa$ in step 4 we obtain, as expected, a direct way to compute a set of generators for the equivariants under the whole group $\G$ from the knowledge of equivariants under the subgroup $\Gmais$. This is formalized below:

\begin{proposition} \label{prop rflexa} Let $\G$ be a compact Lie group acting on $V$ and on $W$ and $\{H_{ij}=\tilde{u}_i H_j, 0 \leq i \leq s , 0\leq j \leq r\}$ a generator set of $\pvwflexamais$ as a $\pg$-module given by step $3$ in Algorithm~$\ref{algoritmo das apl geta equivariantes}$).   Then
	$$\{\Rflexa(H_{i,j}),0 \leq i \leq s , 0\leq j \leq r \}$$
	generates $\pvwflexa$ as a $\pg$-module.
\end{proposition}
\begin{proof}
	Let  $g \in \pvwflexa \subset \pvwflexamais$. Then
	$g=\sum_{i,j}^{s,r} p_{ij}H_{ij}, \ \ \ p_{ij} \in \pg,$  $0 \leq i \leq s$ e $0 \leq j \leq r$.
	Since $\Rflexa$ is a $\pg$-homomorphism and $\Rflexa(g)=g$, then $$g=\Rflexa(g)=\Rflexa\left(\sum_{i,j}^{s,r} p_{ij}H_{ij}\right)= \sum_{i,j}^{s,r} p_{ij} \Rflexa(H_{ij}).  $$
\end{proof}

\subsection{$\Geta$-equivariant quadratic forms} \label{subseq: equiv quadratic forms}

Let $\G$ be a closed subgroup of $\On(n)$ acting linearly on $\R^n$. In this subsection we just rewrite the Reynolds operators of Subsection~\ref{subsec:generalization} applied to the module of  $\Geta$-equivariant quadratic forms of ordem $n$.  In the present context, the representation on the target is defined from the representation on the source by conjugacy.

We consider the action of $\G$ on $Sym_n$, the space of symmetric matrices of order $n$, given by conjugacy,
$$\label{acao Symn}  (\gamma, B) \mapsto \gamma B \gamma^t, \ \gamma \in \G, \  B \in Sym_n . $$

\begin{definition}
	A quadratic form of order $n$ is a matrix-valued mapping $B:\R^n \rightarrow Sym_n$.
\end{definition}

Following the notation of Subsection~\ref{subsec:generalization},   $\pvwflexa$ is the module of the $\G$-equivariants
$$B (\gamma x)=\gamma B(x) \gamma^t, \ \ \forall \gamma \in \G, \ x \in \R^n\label{gamma equiv}, $$
and for $\eta: \G \rightarrow \z_2=\{\pm 1\}$ a nontrivial group homomorphism,    $\qvwflexa$ is the module of the $\Geta$-equivariants
$$ B(\gamma x)=\eta(\gamma) \gamma B(x) \gamma^t, \ \ \forall \gamma \in \G, \  x \in \R^n. \label{geta equiv} $$

The Reynolds operators
$\Rdflexa, \Sdflexa :\pgflexamais \rightarrow \pgflexamais,$ are now
$$\label{reynods r e s} \Rdflexa (g)(x) =1/2(g(x) + \delta^{-1}g(\delta x) \delta),  \ \
\Sdflexa(g)(x)= 1/2 (g(x) - \delta^{-1}g(\delta x)\delta), $$
for an arbitrary fixed $\delta \in \G \setminus \G_{+}$.

\section{\label{sec: general forms} General forms of  symmetric BDEs}

The aim of this section is to present the algebraic forms of  BDEs symmetric under the compact subgroups of ${\bf O}(2)$   with its standard action on the plane.  These  are derived from  generator sets of the modules of equivariant quadratic  forms on the plane
\eq \label{quadratic form} B(x,y)=\left(
\begin{array}{cc}
	c(x,y) & b(x,y) \\
	b(x,y) & a(x,y) \\
\end{array}
\right). \feq
These modules are  $\pvwflexa$ or $\qvwflexa$ of Subsection~\ref{subseq: equiv quadratic forms} for $n=2$. 
If the group homomorphism $\eta : \G \to \z_2$ is nontrivial we apply Algorithm~\ref{algoritmo das apl geta equivariantes}. For clarification of exposition, for each $\G$ and for each possible $\eta$  we shall denote  $\Geta$ by $\G[\ker \eta]$  when $\eta$ is nontrivial, and by $\G$ otherwise. We also denote $\bg^{\eta}(\G)$ by $\bg[\G, \ker \eta]$ and $\bgflexa^{\eta}(\G)$ by $\bgflexa[\G, \ker \eta]$.

For the computations below we shall use  the  action of (subgroups of) $\On(2)$ on $\R^2\simeq \C$ with the usual semi-direct product of $\SO(2)$ and $\z_2(\kappa)$,  using  complex coordinates,
	$$\label{acao padrao} \theta \cdot z=e^{i \theta}z,\ \ \hbox{and} \ \ \kappa z= \bar{z}, \   \ \theta \in [0, 2\pi], \  z \in \C.$$
In Subsections~\ref{subsec: znzn} and \ref{subsec: other zn} we derive the  general forms of symmetric BDEs under the cyclic group $\z_n$, $n \geq 3$,  and in Subsection~\ref{subsec:z2} the general forms under ${\z}_2$, for all  possible homomorphisms $\eta$. For the other compact subgroups of ${\bf O}(2)$  the computation is similar and shall be omitted. In Subsection \ref{subsection:table}  all general forms are given in Table~\ref{tabela forma gerais}.

\subsection{\label{subsec: znzn}$\z_n$-equivariant quadratic forms}

Here we consider the cyclic group $\z_n$, $n \geq 3$,  with $\eta$ trivial. We compute generators of $\pgflexa(\z_n)$ by computing generators of $\mathcal{M}(\z_n)$, the module of $\z_n$-equivariant matrix-valued mappings  $ \R^2 \to  M_2(\R^2)$, and  projecting onto the space of mappings  $\R^2 \to  Sym_2$ . In complex coordinates we write any element of $\mathcal{M}(\z_n)$ as
\eq z \mapsto \alpha(z) w + \beta(z) \bar{w}, \ \forall w \in {\R^2}, \label{matrizcomplexa} \feq
for functions  $\alpha=\alpha_1 +i \alpha_2$ and $\beta=\beta_1 + i\beta_2$,  with $\alpha_j, \beta_j, j=1,2$, real functions.   Associating it  with the real matrix
$$
M = \left(
\begin{array}{cc}
\alpha_1 + \beta_1 & \beta_2 - \alpha_2 \\
\alpha_2 + \alpha_2 & \alpha_1-\beta_1 \\
\end{array}
\right),
$$
the desired quadratic forms are obtained by the projection
\begin{equation}  \label{matrizsimetrica}
M \mapsto B = 1/2 (M + M^t),
\end{equation}
after imposing the $\z_n$-symmetry condition.
Write ($\ref{matrizcomplexa}$) as
$$ M(z)w=\sum\alpha_{jk}z^j\bar{z}^kw + \sum \beta_{jk}z^j\bar{z}^k\bar{w}, \ \ \ \ \alpha_{jk}, \beta_{jk} \in \C.$$
The equivariance  with respect to $\theta \in \z_n$ gives
\eq \label{1} M(z)w=\sum\alpha_{jk}e^{i\theta(j-k)}z^j\bar{z}^kw + \sum \beta_{jk} e^{i \theta(j-k-2)} z^j\bar{z}^k\bar{w}. \feq
So $\alpha_{jk}=\alpha_{jk}e^{i\theta(j-k)}$ and $\beta_{jk}=\beta_{jk}e^{i\theta(j-k-2)}$ for $\theta={2k\pi}/{n}, \ k =1, \ldots, n$, and  so
\eq \label{equiv 1}\alpha_{jk}=0 \ \mbox{ou} \  j\equiv k\mbox{(mod n)} \ \ \ \mbox{and} \ \ \ \beta_{jk}=0 \ \mbox{ou} \ \ j\equiv k+2\mbox{(mod n)}.\feq
A Hilbert basis for $\p(\z_n)$ is given in  \cite{symmetryinchaos},
$$\{ u_1 = z\bar{z}, u_2 =z^n + \bar{z}^n, u_3 =i(z^n -\bar{z}^n)\}.$$
Factor out $z\bar{z}$ in  (\ref{1})  and use (\ref{equiv 1}) to get
\begin{eqnarray*}
	M(z)w &=& \sum_{j\geq k} \alpha_{jk}(z\bar{z})^k z^{j-k}w + \sum_{j< k} \alpha_{j k} (z\bar{z})^j \bar{z}^{k-j}w  + \sum_{j\geq k} \beta_{jk}(z\bar{z})^k z^{j-k}\bar{w} +  \\ & + & \sum_{j < k} \beta_{jk} (z\bar{z})^j \bar{z}^{j-k}\bar{w} \  = \  \sum c^1_{kl_1}(z\bar{z})^k z^{l_1 n}w + \sum c^2_{jl_2} (z\bar{z})^j \bar{z}^{l_2 n}w + \\ &+&  \sum c^3_{k l_3}(z\bar{z})^k z^{l_3n+2}\bar{w} + \sum c^4_{jl_4} (z\bar{z})^j \bar{z}^{l_4 n -2}\bar{w},
\end{eqnarray*}
where, $c^t \in \C$, $l_t \in \mathbb{N}$, $t=1, ...,4$, \ $l_1, l_2, l_3 \geq0$ and $l_4\geq1$. We now use the  identities
\[ \begin{array}{lll}
z^{ln} &=& (z^n+\bar{z}^n)z^{(l-1)n} - (z \bar{z})^n z^{(l-2)n}\\
\bar{z}^{ln}&= &(z^n+\bar{z}^n)\bar{z}^{(l-1)n} - (z \bar{z})^n \bar{z}^{(l-2)n}\\
z^{ln +2}&= &(z^n+\bar{z}^n)z^{(l-1)n +2} - (z \bar{z})^n z^{(l-2)n+2} \\
\bar{z}^{ln-2}&= &(z^n+\bar{z}^n)\bar{z}^{(l-1)n-2} - (z \bar{z})^n \bar{z}^{(l-2)n-2} \\
\bar{z}^n&= & (z^n + \bar{z}^n)-z^n \\
z^{n+2}&= & (z^n+\bar{z}^n)z^2 - (z \bar{z})^2 \bar{z}^{n-2} \\
\bar{z}^{(l+1)n-2}&= & (z^{ln}+\bar{z}^{ln})\bar{z}^{n-2} - (z \bar{z})^{n-2}z^{(l-1)n+2}
\end{array} \]
to conclude that a set of generators of  $\mathcal{M}_2(\z_n)$ over $\p(\z_n)$ is given by the elements
$$ M_1(z)w=w, \  M_2(z)w=iw,  \ M_3(z)w=z^2\bar{w}, \ M_4(z)w=iz^2\bar{w}, \ M_5(z)w=\bar{z}^{n-2}\bar{w},$$ $$M_6(z)w=i\bar{z}^{n-2}\bar{w}, \ M_7(z)w=z^n w,  \ M_8(z)w=iz^n w.$$
We now apply  the projection  (\ref{matrizsimetrica}) to the elements above to find generators of $\pgflexa(\z_n)$,
$$B_{1}(z)w=w, \ B_3(z)w=z^2\bar{w}, \ B_4(z)w=iz^2\bar{w}, \ B_5(z)w=\bar{z}^{n-2}\bar{w}, \ B_6(z)w= i\bar{z}^{n-2}\bar{w}.$$

\subsection{$\z_n[\z_{n/2}]$-equivariant quadratic forms, for $n \geq 4$ even }
\label{subsec: other zn}

In this case,  $\ker \eta = \z_{n/2}$. From the preceding subsection we extract
$$H_0(z)w=w, H_1(z)w=z^2\bar{w}, H_2(z)w=iz^2\bar{w}, H_3(z)w=\bar{z}^{n/2-2}\bar{w}, H_4(z)w=i\bar{z}^{n/2-2}\bar{w}$$
as generators of $\pgflexa(\z_{n/2})$ over the ring $\p(\z_{n/2})$  whose  Hilbert basis  is
$$\{u_1(z)=z \bar{z}, u_2(z)=z^{n/2} + \bar{z}^{n/2}, u_3(z)=i(z^{n/2} -  \bar{z}^{n/2}) \}, $$
We now apply  Algorithm~\ref{algoritmo das apl geta equivariantes}:
\begin{enumerate}
	\item  Fix $\delta=e^{2 \pi i/n} \in \z_n \setminus \z_{n/2}$;
	\item Generators of $\bg[{\z_n, \z_{n/2}}]$ over $\p(\z_n)$:
	$$\tilde{u}_1(z)=S(u_1)(z)=\frac{1}{2}(z\bar{z}-(e^{{2\pi i}/n}z)(e^{-{2 \pi i}/n}\bar{z}))=0.$$
	$$\tilde{u}_2(z)=S(u_2)(z)=\frac{1}{2}(z^{n/2} + \bar{z}^{n/2} - (e^{\pi i}z^{n/2} + e^{-\pi i}\bar{z}^{n/2})=z^{n/2} + \bar{z}^{n/2}.$$
	$$\tilde{u}_3(z)=S(u_3)(z)=\frac{1}{2}(i(z^{n/2} - \bar{z}^{n/2}) - i((e^{\pi i}z^{n/2} - e^{-2\pi i}\bar{z}^{n/2})= i(z^{n/2} - \bar{z}^{n/2}).$$
	\item Generators of  $\pgflexa[\z_n, \z_{n/2}]$ over $\p(\z_n/2)$: set $\tilde{u}_0(z)=1$,
	$$H_{0j}(z)w=\tilde{u}_0(z)H_j(z)w=H_j(z)w, \ j=0, ...,4;$$
	$$H_{1j}(z)w=\tilde{u}_1(z)H_j(z)w=0, \ j=0, ..., 4;$$
	$$H_{20}(z)w=\tilde{u}_2(z)H_0(z)w=(z^{n/2} + \bar{z}^{n/2})w;$$
	$$H_{21}(z)w=\tilde{u}_2(z)H_1(z)w=(z^{n/2 +2} + (z\bar{z})^2 \bar{z}^{n/2-2})\bar{w};$$
	$$H_{22}(z)w=\tilde{u}_2(z)H_2(z)w=i (z^{n/2 +2} + (z\bar{z})^2 \bar{z}^{n/2-2})\bar{w};$$
	$$H_{23}(z)w=\tilde{u}_2(z)H_3(z)w=(\bar{z}^{n/2 -2} + (z\bar{z})^{n/2-2} z^2)\bar{w};$$
	$$H_{24}(z)w=\tilde{u}_2(z)H_4(z)w=i(\bar{z}^{n/2 -2} + (z\bar{z})^{n/2-2} z^2)\bar{w};$$
	$$H_{30}(z)w=\tilde{u}_3(z)H_0(z)w=i(z^{n/2}- \bar{z}^{n/2})w;$$
	$$H_{31}(z)w=\tilde{u}_3(z)H_1(z)w=i(z^{n/2 +2} -(z\bar{z})^2 \bar{z}^{n/2-2})\bar{w};$$
	$$H_{32}(z)w=\tilde{u}_3(z)H_2(z)w=(-z^{n/2 +2} +(z\bar{z})^2 \bar{z}^{n/2-2})\bar{w};$$
	$$H_{33}(z)w=\tilde{u}_3(z)H_3(z)w=i (-\bar{z}^{n-2} + (z\bar{z})^{n/2-2}z^2)\bar{w};$$
	$$H_{34}(z)w=\tilde{u}_3(z)H_4(z)w=(\bar{z}^{n-2} - (z\bar{z})^{n/2-2}z^2)\bar{w},$$
	which, as an intermediate step, we simplify to the reduced list
	$$ H_{00}(z)w=w, H_{01}(z)w=z^2\bar{w}, H_{02}(z)w=iz^2\bar{w}, H_{03}(z)w=\bar{z}^{n/2-2}\bar{w},$$
	$$ H_{04}(z)w=i\bar{z}^{n/2-2}\bar{w}, H_{20}(z)w=(z^{n/2} - \bar{z}^{n/2}) w, H_{21}(z)w=z^{n/2+2}\bar{w}, H_{22}(z)w=iz^{n/2+2}\bar{w},$$ 
	$$ H_{23}(z)w=\bar{z}^{n-2}\bar{w},  H_{24}(z)w=i\bar{z}^{n-2}\bar{w},
	H_{30}(z)w=i(z^{n/2} -\bar{z}^{n/2})w. $$
	
	\item  Generators  of $\bgflexa[\z_n, \z_{n/2}]$ over $\p(\z_n)$:
	$$\tilde{H}_{00}(z)w=\tilde{H}_{01}(z)w=\tilde{H}_{02}(z)w=0;$$
	$$\tilde{H}_{03}(z)w=\bar{z}^{n/2-2}\bar{w};$$
	$$\tilde{H}_{04}(z)w=i\bar{z}^{n/2-2}\bar{w};$$
	$$\tilde{H}_{20}(z)w=(z^{n/2} + \bar{z}^{n/2}) w;$$
	$$\tilde{H}_{21}(z)w=z^{n/2+2}\bar{w};$$
	$$\tilde{H}_{22}(z)w=iz^{n/2+2}\bar{w};$$
	$$\tilde{H}_{23}(z)w=\tilde{H}_{24}(z)w=0;$$
	$$\tilde{H}_{20}(z)w=i(z^{n/2} - \bar{z}^{n/2}) w.$$
\end{enumerate}
Therefore, $\bgflexa[\z_n, \z_{n/2}]$ is the $\p(\z_n)$-module generated by $$\tilde{B}_1(z)w=\bar{z}^{n/2-2}\bar{w}, \ \tilde{B}_2(z)w=i \bar{z}^{n/2 -2}\bar{w}, \ \tilde{B}_3(z)w=(z^{n/2} + \bar{z}^{n/2})w,$$
$$ \tilde{B}_4(z)w=z^{n/2+2}\bar{w}, \ \tilde{B}_5(z)w=iz^{n/2+2}\bar{w}, \ \tilde{B}_6(z)w=i(z^{n/2} - \bar{z}^{n/2}) w.$$

\subsection{${\z}_2$-equivariant quadratic forms} \label{subsec:z2}

Let ${\z}_2$ be the group generated by the reflection $\kappa_x$ on the $x$-axis. First we consider $\eta : {\z}_2(\kappa_x) \to \z_2$ trivial. Imposing the ${\z}_2$-equivariance to (\ref{quadratic form}) gives
$$ \label{eq: z2 equi} \left(
\begin{array}{cc}
c(x,-y) & b(x,-y) \\
b(x,-y) & a(x,-y) \\
\end{array}
\right)=\left(
\begin{array}{cc}
c(x,y) & -b(x,y) \\
-b(x,y) & a(x,y) \\
\end{array}
\right).
$$
This is to say that $a$ and $c$ are ${\z}_2[ {\bf 1}]$ -equivariant and $b$ is ${\z}_2$-anti-invariant. Therefore, the generators of $\pgflexa({\z}_2)$ under $\p({\z}_2)$ are
$$(x,y) \mapsto \left(
\begin{array}{cc}
1 & 0 \\
0 & 0 \\
\end{array} \right),  \ (x,y) \mapsto \left(
\begin{array}{cc}
0 & 0 \\
0 & 1 \\
\end{array}\right),  \ (x,y) \mapsto \left(
\begin{array}{cc}
0 & y \\
y & 0 \\
\end{array} \right).$$

Assume now $\eta$ nontrivial, so  $\ker \eta = {\bf 1}$.  Imposing the ${\z}_2[{\bf 1}]$-equivariance to (\ref{quadratic form}) gives
$$ \label{eq1: z2 equi} \left(
\begin{array}{cc}
c(x,-y) & b(x,-y) \\
b(x,-y) & a(x,-y) \\
\end{array}
\right)=\left(
\begin{array}{cc}
-c(x,y) & b(x,y) \\
b(x,y) & -a(x,y) \\
\end{array}
\right).
$$
Hence $b$ is ${\z}_2$-invariant and the functions $a$ and $c$ are ${\z_2}[{\bf 1}]$-equivariant. Therefore, the generators for $\pgflexa({\z_2},{\bf 1})$ under $\p({\z}_2)$ are
$$(x,y) \mapsto \left(
\begin{array}{cc}
0 & 1 \\
1 & 0 \\
\end{array} \right),  \ (x,y) \mapsto \left(
\begin{array}{cc}
y & 0 \\
0 & 0 \\
\end{array}\right),  \ (x,y) \mapsto \left(
\begin{array}{cc}
0 & 0 \\
0 & y \\
\end{array} \right).$$

\subsection{ \label{subsection:table}  Summarizing table and illustrations}
In this subsection we present  the general forms of symmetric  quadratic differential 1-forms $ \omega = a(x,y)dy^2+2b(x,y)dxdy +c(x,y)dy^2$ under compact subgroups of $\mathbf{O}(2)$. 
Table~1 shows each group $\G$ with all possible values of $\ker \eta$, denoted by $\Gamma[\ker \eta]$. Following the previous notation, when $\eta$ is trivial the group is denoted simply by $\Gamma$. Also, $\D_{n}(\kappa_x)$ and $\D_{n}(\kappa_y)$ shall denote  the dihedral groups generated by the rotation of angle $2 \pi/n$ and by the reflections with respect to the $x$-axis or $y$-axis, respectively.

In \cite{brucetari} the authors consider BDEs whose discriminant function $\delta=b^2 - ac$  is of Morse type. In this case, the discriminant set is a pair of transversal straight lines by the origin or the origin itself. They prove that these BDEs are topologically equivalent to their linear part. We remark that all the normal forms that they obtain must be equivariant under a finite symmetry group. In fact, it follows from Table~1 that there are no linear BDEs with infinite group of symmetries.  As it appears in \cite{brucetari},   the Morse condition is given in terms of the coefficients of the linear part of the smooth functions $a, b$ and $c$. More precisely, if we write $a=a_1x + a_2y + {\bf o}(2)$, $b=b_1x + b_2y + {\bf o}(2)$ and $c=c_1x + c_2y + {\bf o}(2)$, then the condition is  
\begin{equation}
\label{morse} (c_2a_1-c_1a_2)^2 -4(b_2a_1-b_1a_2)(c_2b_1-c_1b_2)\neq 0.
\end{equation}
From Table~1, the possible symmetry groups of BDEs whose linear parts satisfy (\ref{morse}) are
\begin{equation}  \label{eq: class 1}
\z_3, \ \z_6[\z_3], \ \D_3, \D_3[\z_3],  \ \D_6[\D_3]
\end{equation}
or  
\begin{equation} \label{eq: class 2}
\z_2, \   \z_2[{\bf 1}], \    \z_2\times \z_2[\z_2(\kappa_x)]. 
\end{equation}

Recall from Remark \ref{rmk: inclusion1} that the set of all  symmetries of a BDE is at most the symmetry group $\Sigma(\Delta)$ of the discriminant set. Hence, for the Morse cases it follows that if $\Delta$ is the origin, then  the possible  nontrivial symmetry groups  are the ones in (\ref{eq: class 1}), whereas the groups listed in (\ref{eq: class 2})  are the possible groups  when the discriminant set is a pair of transversal straight lines. We also point out that  the finiteness of the symmetry group  also holds for equations with constant coefficients. A classification of these two types of BDEs is done in \cite{ManTemp1}, including an analysis of the corresponding group of symmetries of the equation with possible number of invariant lines in the associated configuration.

\begin{remark}
	The symmetry group of the configuration shown in Fig.1$(c)$ is $\D_6[\D_3(\kappa_y)]$, whose quadratic form $(y, x, -y)$
	appears in Table $1$  by interchanging the variables $x$ and $y$ and taking $p_1 \equiv 1$ and $p_2 \equiv p_3 \equiv 0$ in the general form for the group $\D_6[\D_3(\kappa_x)].$ Similarly, the symmetry group of the configurations in Fig.~\ref{Figura 1}$(a)$ and $(b)$  is $\z_2 \times \z_2[\z_2(\kappa_y)]$, whose quadratic forms appear from the data for $\z_2\times \z_2[\z_2(\kappa_x)]$ in Table~1 by interchanging $x$ and $y$   and taking $p_1\equiv p_2 \equiv1, p_3\equiv -1$, and $p_1\equiv1, p_2 \equiv \frac{1}{4}, p_3 \equiv -1$, respectively.
\end{remark}

\begin{table}[!h] \label{table}
	\begin{center}
	\begin{tabular}{|c|c|l|}
		\hline
		$\Gamma[\ker \eta]$        & $\ker \lambda$          &  General form	  \\ \hline
		
		&              & $a=p_1 + (y^2-x^2)p_2 + 2xyp_3;$ \\
		$\SO(2)$ & $\SO(2)$  	& $b= 2xyp_2 + (x^2-y^2)p_3;$ \\ 
		&				& $c= p_1 + (x^2-y^2)p_2 - 2xyp_3,$ \\
		&         &      $  p_i \in \p(\SO(2)), i=1,2,3.$ \\ \hline
		
		$\On(2)$ &   $\SO(2)$   & $a= p_1 + (y^2-x^2)p_2; \ b=2xyp_2; $ \\ &   & $  c=p_1 + (x^2-y^2)p_2, \ p_i \in \p(\On(2)), i=1,2.$ \\ \hline
		
		                  &           & $a= 2xyp;$ \\
		$\On(2)[\SO(2)]$  & $\On(2)$  & $b= (x^2-y^2)p;$ \\ 
		                  &  & $  c= -2xyp, \p \in \p(\On(2)). $ \\ \hline			
		
		&  		 &	$a= p_1 + (y^2-x^2)p_2 + 2xyp_3 -A_1p_4 -A_2p_5;$ \\
		$\z_n, $ & $\z_n$    &  $b= 2xyp_2 + (x^2-y^2)p_3 + A_1p_5 -A_2p_4;$ \\ 
		  $n \geq 3$   &      &  $c= p_1 + (x^2-y^2)p_2 -2xyp_3 +A_1p_4 + A_2p_5, $ \\
		&           &  $ p_i \in \p(\z_n ), i=1,..., 5.$  \\ \hline
		
		&                    & $a= -A_3p_1 - A_4p_2 + A_5p_3 -A_7p_4 + A_8p_5 + A_6p_6;$ \\ 	                              
		$\z_n[\z_{n/2}], $ & $\z_{n/2}$ & $b= -A_4p_1 + A_3p_2 + A_8p_4 + A_7p_5; $ \\ 
		$ n \geq 4$ even &  & $c= A_3p_1 + A_4p_2 + A_5p_3 +A_7p_4 - A_8p_5 + A_6p_6,$ \\ 
		&                    & $ p_i \in \p(\z_n),  i=1, ..., 6.  $ \\ \hline

		&         &  $a= p_1 + (y^2-x^2)p_2 - A_1p_3;$ \\
		$\D_n,  $ & $\z_n$  &  $b= 2xyp_2 - A_2p_3; $ \\
		$n \geq 3 $ &   &  $c= p_1 + (x^2-y^2)p_2 + A_1p_3, \ p_i \in \p(\D_n), i=1,2,3.$ \\ \hline
		
		&        & $a= 2xyp_1 -A_2p_2 +A_9p_3;$ \\
		$\D_n[\z_n], $ & $\D_n$ & $b= (x^2-y^2)p_1 + A_1p_2; $ \\
		  $n \geq 3$      &        & $c=-2xyp_1 + A_2p_2 +A_9p_3,  \ p_i \in \p(\D_n), i=1, 2, 3.$ \\ \hline
		
		&                                & $a= -A_3p_1 +A_5p_2 -A_7p_3; $ \\
		$\D_n[\D_{n/2}(\kappa_x)], $ & $\D_{n/2}(\kappa_y)]$  & $b= -A_4p_1 +A_8p_3; $ \\ 
	    $ n \geq 4$ even	&     & $c= A_3p_1 +A_5p_2 + A_7p_3, \ p_i \in \p(\D_n), i=1, 2, 3.$ \\   \hline
		
		$\z_2$ & ${\bf 1}$ &  $a= p_1; \ b= yp_2; \ c=p_3, \ p_i \in \p(\z_2), i=1, 2,3.$ \\ \hline
		
		$\z_2[{\bf 1}]$ & $\z_2$ & $a= y p_1; \ b= p_2; \ c=yp_3, \  p_i \in \p(\z_2), i=1, 2, 3.$ \\ \hline
		
		$\z_2 \times \z_2$ & $\z_2(-I)$  & $a=p_1; \ b= xyp_2; \ c=p_3, \   p_i \in \p(\z_2 \times \z_2), i=1, 2, 3.$ \\ \hline
		
		$\z_2 \times \z_2[\z_2(-I)]$ & $\z_2 \times \z_2 $ & $a=xyp_1; \ b= p_2; \ c=xyp_3,  p_i \in \p(\z_2 \times \z_2), i=1, 2, 3.$ \\ \hline
		
		$\z_2 \times \z_2[\z_2(\kappa_x)]$ & $\z_2(\kappa_y)$ & $a= xp_1; \ b= yp_2; \  c=xp_3,  p_i \in \p(\z_2 \times \z_2), i=1, 2, 3.$ \\ \hline
			
			\multicolumn{3}{|l|}{$A_1= \hbox{Re}(z^{n-2}), \ A_2= \hbox{Im}(z^{n-2}), \ A_3= \hbox{Re}(z^{n/2-2}), \ A_4= \hbox{Im}(z^{n/2-2}), \ A_5= \hbox{Re}(z^{n/2}),$ } \\ \hline
			\multicolumn{3}{|l|}{$A_6= \hbox{Im}(z^{n/2}), \  A_7= \hbox{Re}(z^{n/2+2}), \ A_8= \hbox{Im}(z^{n/2+2}),  \ A_9= \hbox{Im}(z^n). $} \\ \hline
	
		\end{tabular}
	\end{center}
			\caption{General forms of equivariant quadratic differential forms on the plane under closed subgroups of {\bf O}$(2)$.}
			\label{tabela forma gerais}
\end{table}




We finish this paper with an example of each symmetry type  given in Table \ref{tabela forma gerais}. Let us point out that some of these configurations  can be realized for example as lines of curvatures or as asymptotic lines of surfaces immersed in $\R^3$  or $\R^4$, whereas some others cannot. This is an interesting issue in differential geometry that we have started to investigate in presence of symmetries. For the context without symmetry we cite \cite{hopf,livroronaldo,gutierrezsoto}.

\begin{figure}[!h]
	\centering
	\subfigure[$(1 +y^2 -x^2+2xy, x^2-y^2 +2xy, 1+x^2 -y^2 -2xy)$][\label{so2so2}]{
		\includegraphics[scale=0.5]{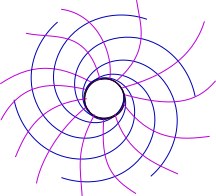}}
	\qquad 
	\subfigure[$(y^2-x^2, 2xy, x^2-y^2)$][\label{o2o2}]{
		\includegraphics[scale=0.5]{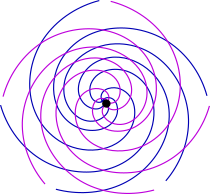}}
	\qquad 
	\subfigure[$(2xy, x^2-y^2, -2xy )$][\label{o2so2}]{
		\includegraphics[scale=0.5]{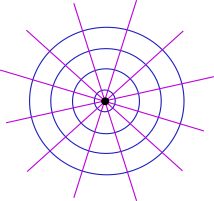}}
	\caption{Configurations with symmetry (a) $\mathbf{SO}(2)$, (b) $\mathbf{O}(2)$ and (c) $\mathbf{O}(2)[\mathbf{SO}(2)]$.}
	\label{fig01}
\end{figure}

For   $\mathbf{SO}(2)$, we choose  $p_1 \equiv p_2 \equiv p_3\equiv1$  in Table~\ref{tabela forma gerais}, so that the differential form is
$$(1 +y^2 -x^2+2xy, x^2-y^2 +2xy, 1+x^2 -y^2 -2xy).$$
The homomorphism $\lambda$ is trivial and the discriminant function is  $\mathbf{O}(2)$-invariant given by $$\delta(x,y)=2(x^2 +y^2)^2 -1.$$
This is illustrated in Fig.~\ref{so2so2}.

For  $\mathbf{O}(2)$, we choose $p_1\equiv0, p_2\equiv1$, so the differential form is
$$(y^2-x^2, 2xy, x^2-y^2).$$
The homomorphism $\lambda$ is such that  $\ker \lambda = \SO(2)$ and the discriminant function is $\mathbf{O}(2)$-invariant  given by $$\delta(x,y)=(x^2 +y^2)^2.$$
This is illustrated in Fig.~\ref{o2o2}.

The configuration in Fig.~\ref{o2so2} is  $\mathbf{O}(2)[\mathbf{SO}(2)]$-symmetric, whose quadratic form has been chosen by taking  $p\equiv1$ in Table~1, that is, $$(2xy, x^2-y^2, -2xy ).$$ The homomorphism $\lambda$ is trivial and the discriminant function is the  $\mathbf{O}(2)$-invariant given by $$\delta(x,y)=(x^2 +y^2)^2. $$

\begin{figure}[!h]
	\centering
	\subfigure[()][\label{znzn}]{
		\includegraphics[scale=0.5]{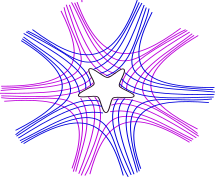}}
	\qquad 
	\subfigure[$()$][\label{znzn2}]{
		\includegraphics[scale=0.5]{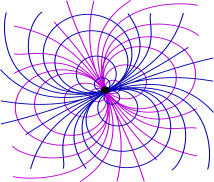}}
	\caption{Configurations with symmetry group given by (a) $\mathbf{Z}_5$ and (b) $\mathbf{Z}_4[\mathbf{Z}_2]$ .}
	\label{fig02}
\end{figure}

We now consider  $\mathbf{Z}_5$     taking  $p_1\equiv p_2 \equiv p_5 \equiv1$ and $p_3\equiv p_4 \equiv p_5 \equiv 0$ in Table~\ref{tabela forma gerais}, so that the differential form is
$$(1 +y^2 -x^2 -3x^2y +y^3, 2xy + x^3 -3xy^2, 1 + x^2 -y^2 +3x^2y -y^3).$$
The homomorphism  $\lambda$ is necessarily trivial. The discriminant function is the $\mathbf{Z}_5$-invariant given by
$$\delta(x,y)=(x^2 +y^2)^3 +10x^4y -20x^2y^3+2y^5 +(x^2 +y^2)^2 -1.$$
The picture for this case is shown in Fig. \ref{znzn}. The star shape of the discriminant set is in fact $\z_5$-symmetric without reflectional symmetries, as it is easily checked by direct calculation.

Fig.~\ref{znzn2} is a $\mathbf{Z}_{4}[\mathbf{Z}_{2}]$ case, considering  $p_1\equiv p_2\equiv p_4 \equiv p_5\equiv 1$ and $p_3 \equiv p_6 \equiv0 $ in Table~\ref{tabela forma gerais}, so that the differential form is
$$( -x^4 +6x^2y^2 -y^4+4x^3y -4xy^3,  x^4 -6x^2y^2 +y^4+4x^3y -4xy^3, x^4 -6x^2y^2 +y^4-4x^3y +4xy^3).$$
The homomorphism $\lambda $ must be  such that $\ker \lambda = \z_2.$ The discriminant set  is just the origin, given as the zero set of (the  $\mathbf{O}(2)$-invariant)
$$\delta(x,y)=2(x^2 +y^2)^4.$$

\begin{figure}[!h]
	\centering
	\subfigure[()][\label{dndn}]{
		\includegraphics[scale=0.5]{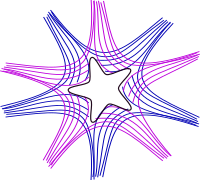}}
	\qquad 
	\subfigure[$(y^2-x^2, 2xy, x^2-y^2)$][\label{dnzn}]{
		\includegraphics[scale=0.5]{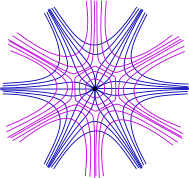}}
	\qquad 
	\subfigure[$(2xy, x^2-y^2, -2xy )$][\label{dndnd}]{
		\includegraphics[scale=0.5]{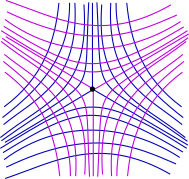}}
	\caption{Configurations with symmetry groups $\mathbf{D}_5$, $\mathbf{D}_6[\mathbf{Z}_6]$ and $\mathbf{D}_6[\mathbf{D}_3(\kappa_x)]$.}
	\label{fig03}
\end{figure}

For $\mathbf{D}_5$ in Table~1, we take $ p_1\equiv p_2 \equiv p_3 \equiv1$, so that the differential form is
$$(1 +y^2 -x^2 -x^3 +3xy^2, 2xy -3x^2y +y^3, 1 -y^2 +x^2 +x^3 -3xy^2).$$
In this case,   $\ker \lambda= \mathbf{Z}_5 $ and the discriminant function is the  $\mathbf{D}_5$-invariant given by
$$\delta(x,y)=(x^2 +y^2)^3 +2x^5 -20x^3y^2+10xy^4 +(x^2 +y^2)^2 -1.$$
This is illustrated in Fig.~\ref{dndn}.

We now consider $\mathbf{D}_6[\mathbf{Z}_6]$  choosing $ p_1\equiv p_2 \equiv 1$ and $p_3 \equiv 0$ in Table~\ref{tabela forma gerais}, so that the form is
$$( 2xy -4x^3y +4xy^3, x^2-y^2 +x^4-6x^2y^2 +y^4, -2xy +4x^3y -4xy^3).$$
In this case  $\lambda $ is trivial and the discriminant function is   $\mathbf{D}_6$-invariant and given by
$$\delta(x,y)=(x^2 +y^2)^4+2x^6-30x^4y^2+30x^2y^4 +2y^6 + (x^2 +y^2)^2.$$
The picture is given in Fig.~\ref{dnzn}.

We now turn to $\mathbf{D}_6[\mathbf{D}_{3}(\kappa_x)]$  taking $ p_1\equiv1$ and $p_2 \equiv p_3 \equiv 0$ in Table~1, so that the form is
$$( -x, -y, x).$$
In this case,   $\ker \lambda = \mathbf{D}_3(\kappa_y)$ and the discriminant set is the origin, given by the zero set of
$$\delta(x,y)=x^2 + y^2.$$
The picture is given in Fig.~\ref{dndnd}.

\begin{figure}[!h]
	\centering
	\subfigure[$()$][\label{kxkx}]{
		\includegraphics[scale=0.5]{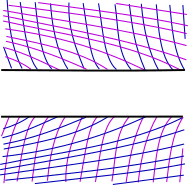}}
	\qquad 
	\subfigure[$(2xy, x^2-y^2, -2xy )$][\label{kxki}]{
		\includegraphics[scale=0.5]{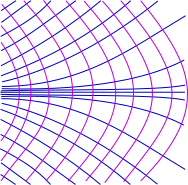}}
	\caption{Configurations with symmetry groups $\mathbf{Z}_2$ and  $\mathbf{Z}_2[{\bf 1}]$.}
	\label{fig04}
\end{figure}

Consider now  $\mathbf{Z}_2$ for $p_1 \equiv p_2 \equiv p_3 \equiv 1$ from Table~1, so that the form is
$$( 1, y, 1).$$
We have   $\ker \lambda = {\bf 1}$ and the discriminant function is $\mathbf{Z}_2$-invariant and given by
$$\delta(x,y)=y^2-1.$$
See the illustration of this case in Fig.~\ref{kxkx}.

Fig.~\ref{kxki} is a $\mathbf{Z}_2[{\bf 1}]$ case, for which we have chosen $p_1 \equiv p_2 \equiv 1$ and $p_3\equiv-1$ in Table \ref{tabela forma gerais} , so that the form is
$$( y, 1, -y).$$
The homomorphism $\lambda $ is trivial and the discriminant function is  $\mathbf{Z}_2$-invariant and given by $$\delta(x,y)=y^2 +1.$$

\begin{figure}[!h]
	\centering
	\subfigure[()][\label{z2z2z2z2}]{
		\includegraphics[scale=0.5]{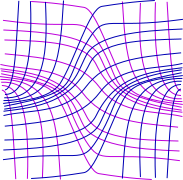}}
	\qquad 
	\subfigure[$()$][\label{z2z2zi}]{
		\includegraphics[scale=0.5]{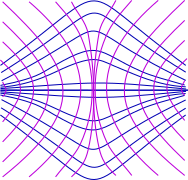}}
	\qquad 
	\subfigure[$()$][\label{z2z2z2k}]{
		\includegraphics[scale=0.5]{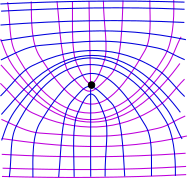}}
	\caption{ Configurations with symmetry groups $\mathbf{Z}_2 \times \mathbf{Z}_2, \mathbf{Z}_2 \times \mathbf{Z}_2[\mathbf{Z}_2(-I)]$ and  $\mathbf{Z}_2 \times \mathbf{Z}_2 [\mathbf{Z}_2(\kappa_x)]$.}
	\label{fig05}
\end{figure}

For $\mathbf{Z}_2 \times \mathbf{Z}_2$ in Table \ref{tabela forma gerais}, we take $p_1 \equiv p_2 \equiv 1$ and $p_3\equiv -1$, so that the differential form is  $$( -x, -y, x).$$  In this case $\ker \lambda = \mathbf{Z}_2$ and  the discriminant function is the $\mathbf{Z}_2 \times \mathbf{Z}_2$-invariant  given by $$\delta(x,y)=x^2y^2+1.$$ This is illustrated in Fig. \ref{z2z2z2z2}.

We now consider  $\mathbf{Z}_2\times \mathbf{Z}_2[\mathbf{Z}_2(-I)]$ choosing $ p1\equiv p_2 \equiv 1$ and $p_3 \equiv-1$ in Table \ref{tabela forma gerais}, so that the form is $$( xy, 1, -xy).$$ In this case $\lambda $ is trivial and the discriminant function is $\mathbf{Z}_2\times \mathbf{Z}_2$-invariant and given by $$\delta(x,y)=x^2y^2+1.$$ The picture is given in Fig. \ref{z2z2zi}.

Finally, consider $\mathbf{Z}_2\times \mathbf{Z}_2[\mathbf{Z}_2(\kappa_x)]$ taking $p_1 \equiv1, p_2 y^2$ and $p_3 \equiv -1$ in Table \ref{tabela forma gerais}, so that differential form is  $$( x, y^3, -x).$$  In this case $\ker \lambda = \mathbf{Z}_2(\kappa_y)$ and the discriminant function is  given by $$\delta(x,y)=x^2 +y^6.$$ See the illustration of this case in Fig. \ref{z2z2z2k}. \\

{\bf Acknowledgements}:	Research of P. T.  was supported by CAPES Grant 8474758/D.

\end{document}